\documentclass[11pt]{amsart}
\usepackage{amsmath,amssymb,latexsym,amsthm,enumerate}
\usepackage{hyperref} %To be able to click on the links in the paper
\theoremstyle{plain}
\newtheorem{theorem}{Theorem}[section]
\newtheorem{proposition}[theorem]{Proposition}
\newtheorem{lemma}[theorem]{Lemma}

\theoremstyle{definition}

\newtheorem{remark}[theorem]{Remark}
\newtheorem{remarks}[theorem]{Remarks}
\theoremstyle{remark}

\newcounter{numpar}[section]
\renewcommand*{\thenumpar}{\thesection.\arabic{numpar}}
\newcommand*{\newpar}[1]{\refstepcounter{numpar}\medskip\textbf{\thenumpar.#1}}
\numberwithin{equation}{section}

\newcommand*{\wt}{\widetilde}
\newcommand*{\ol}{\overline}

\newcommand*{\dd}{d}     %For dx, dy, etc.
\newcommand*{\eps}{\varepsilon}

\newcommand*{\cF}{\mathcal F}

\newcommand*{\bbR}{\mathbb R}

\newcommand*{\fM}{\mathfrak M}
\newcommand*{\EE}{\mathsf E}
\newcommand*{\PP}{\mathsf P}

\newcommand*{\la}{\langle}
\newcommand*{\ra}{\rangle}
\newcommand*{\loc}{{\mathrm{loc}}}
\DeclareMathOperator{\Law}{Law}

\begin{document}
\title[Convergence of Integral Functionals]{Convergence of Integral Functionals of One-Dimensional Diffusions}

\author{Aleksandar Mijatovi\'{c}}
\address{Department of Statistics, University of Warwick, UK}
\email{a.mijatovic@warwick.ac.uk}

\author{Mikhail Urusov}
\address{Institute of Mathematical Finance, Ulm University, Germany}
\email{mikhail.urusov@uni-ulm.de}

%\thanks will become a 1st page footnote.
\thanks{We are grateful to Hans-J\"urgen Engelbert for valuable discussions 
and advice on this topic.  This research was supported
by the LMS grant R.STAA.3037 and through the program ``Research in Pairs''
by the Mathematisches Forschungsinstitut Oberwolfach in 2011.
MU would like to thank the Statistics Department at Warwick, where a part of this work
was carried out, for hospitality and a productive research atmosphere.}



\begin{abstract}
In this expository paper we describe the pathwise behaviour of the 
integral functional
$\int_0^t f(Y_u)\,\dd u$ 
for any
$t\in[0,\zeta]$,
where 
$\zeta$ 
is (a possibly infinite) exit time of a one-dimensional diffusion process
$Y$
from its state space,
$f$
is a nonnegative Borel measurable function
and the coefficients of the SDE solved by 
$Y$
are only required to satisfy weak local integrability conditions.
Two proofs of the deterministic characterisation of the convergence 
of such functionals are given: the problem is reduced in two different ways 
to certain path properties of Brownian motion where either the Williams theorem 
and the theory of Bessel processes or the first Ray-Knight theorem can be applied to 
prove the characterisation.
\end{abstract}

\maketitle

%\tableofcontents

%========================================================================================
\section{Introduction}
\label{sec:i}

The \textit{Engelbert--Schmidt zero-one law}
states that for a Brownian motion 
$B$
and any nonnegative Borel function 
$f$ the following statements are equivalent:
\begin{enumerate}[(a)]
\item $\PP\left(\int_0^tf(B_s)\dd s<\infty\>\text{for all}\>t\in[0,\infty)\right)>0$;
\label{eq:a}
\item $\PP\left(\int_0^tf(B_s)\dd s<\infty\>\text{for all}\>t\in[0,\infty)\right)=1$;
\label{eq:b}
\item the function $f$ is locally integrable on $\bbR$.
\label{eq:c}
\end{enumerate}
This important property has a plethora of applications. For example it constitutes an 
important step in the Engelbert--Schmidt construction 
of weak solutions of one-dimensional 
SDEs. 
%This application and 
The proof of the zero-one law 
can be found in monograph~\cite[Ch.~3]{KaratzasShreve:91}
or original article~\cite{EngelbertSchmidt:81}.
%,~\cite{EngelbertSchmidt:91}.
Note that the  equivalences 
between~\eqref{eq:a},~\eqref{eq:b} and~\eqref{eq:c}
do not contain any information about the behaviour of the integral when 
local integrability of the function 
$f$
fails on a subset of 
$\bbR$.
The precise description of the explosion time of this integral functional
was given in~\cite[Lem.~1]{EngelbertSchmidt:85}.

In this paper we investigate a related problem of 
the convergence of the integral functional
\begin{equation*}
%\label{eq:set7}
\int_0^t f(Y_u)\,\dd u,\quad t\in[0,\zeta],
\end{equation*}
of a one-dimensional $J$-valued diffusion 
$Y$
that solves an SDE %~\eqref{eq:set1}, 
up to an exit time
$\zeta$.
The coefficients 
%$\mu,\sigma:J\to\bbR$
of the SDE %~\eqref{eq:set1}
are required to satisfy only some weak
local integrability conditions on the open interval 
$J$
(see~\eqref{eq:set2} and~\eqref{eq:set3}
for the precise form of the Engelbert--Schmidt conditions satisfied by 
the coefficients)
and the function 
$f\colon J\to[0,\infty]$
is assumed to be Borel measurable.
The main results in this paper (see Theorems~\ref{th:set1},~\ref{th:set2} 
and~\ref{th:set3})
study the integral functional as a process, identify the 
stopping time after which the integral explodes,
and give a deterministic criterion for 
the convergence of the integral functional
at this stopping time.
It turns out that this stopping time 
is the first time the process 
$Y$
hits the set where the local integrability 
condition,
analogous to~\eqref{eq:c} above, fails.

The proof of the results consists of two steps. The first
step, which uses only basic properties of diffusion processes 
and their local times, reduces the original problem 
to a question about the convergence of an integral functional 
of Brownian motion.  In the second step we give 
two proofs for the characterisation of the convergence of this integral functional
of Brownian motion: (i)~Williams' theorem (see~\cite[Ch.~VII, Cor.~4.6]{RevuzYor:99})
and the result on integral functionals of Bessel processes
from Cherny~\cite{Cherny:01} are applied; 
(ii)~a direct approach based on 
the first Ray-Knight theorem (see~\cite[Ch.~XI, Th.~2.2]{RevuzYor:99}) 
is followed.
%similar to what can be found
%in the literature on this topic,
%inspired by a trick in the proof 
%of Theorem~1.4 in Delbaen and Shirakawa~\cite{DelbaenShirakawa:02b}
%and an application of the first Ray-Knight theorem 
%(see \cite[Ch.~XI, Th.~2.2]{RevuzYor:99}) yield the 
%characterisation. 

In~\cite{EngelbertTittel:00} Engelbert and Tittel investigate the convergence 
of the integral functionals of the form
$\int_0^tf(X_s)\,\dd s$,
where 
$f$
is a nonnegative Borel function and 
$X$ 
a strong Markov continuous local martingale.
The analytic condition that characterises the convergence 
of the integral functionals is given in terms of the speed 
measure of 
$X$.
The diffusion 
$Y$
considered in this paper is not necessarily a local martingale. 
However, the process
$s(Y)$,
where 
$s$
is the scale function of 
$Y$,
is
%puts us in the setting of~\cite{EngelbertTittel:00}
and our characterisation theorems can be deduced from the ones
in~\cite{EngelbertTittel:00}. The proofs of
the main results in~\cite{EngelbertTittel:00} are based on
Lemma~3.1 in~\cite{EngelbertTittel:00},
attributed to Jeulin~\cite{Jeulin:80},
%and implies 
%a result of Assing~\cite{Assing:94} 
%(cf. Assing and Schmidt~\cite[Lem.~A1.8]{AssingSchmidt:98}),
%namely
%is applied to obtain a 
which, together with 
the Ray-Knight theorem, % (see~\cite[Ch.~XI, Th.~2.2]{RevuzYor:99}), 
implies a version of a zero-one law for Brownian local 
time integrated 
in the space variable against a measure on 
$\bbR$
(see also Assing~\cite{Assing:94}).
%This zero-one law and the application of the Ray-Knight theorem
%are used to prove the characterisation theorem in~\cite{EngelbertTittel:00}.
This zero-one law for Brownian local time 
%mentioned in the 
%previous paragraph 
is closely related to key
Lemma~\ref{lem:bc1}.
The result in
Lemma~\ref{lem:bc1}
first appeared implicitly 
in the paper of Engelbert and Schmidt~\cite[p.~225--226]{EngelbertSchmidt:87}
and was stated explicitly by Assing and Senf~\cite[Lem.~2]{AssingSenf:97}.
The proofs in~\cite{EngelbertSchmidt:87}
and~\cite{AssingSenf:97} rest on an application of
the Ray-Knight theorem and Shepp's~\cite{Shepp:66} dichotomy result 
for Gaussian processes.  This dichotomy argument was later 
replaced by the abstract but elementary lemma of Jeulin~\cite{Jeulin:80}
(see Assing and Schmidt~\cite[Lem.~A1.7]{AssingSchmidt:98}).

%Given the results in~\cite{EngelbertTittel:00},
%the work in the present paper is expository in nature 
%with the mathematical contribution confined to 
%Theorem~\ref{th:set1} and
%a new proof of 
%Lemma~\ref{lem:bc1}. 
The emphasis in the present paper 
is on understanding the pathwise behaviour of the 
integral functionals of one-dimensional diffusions
directly from the pathwise properties of Brownian motion.
Our proofs are short and are based on a simple direct approach 
which reduces the problem to Brownian motion where either
the Williams theorem and Cherny's results from~\cite{Cherny:01}
or an idea from  Delbaen and Shirakawa~\cite{DelbaenShirakawa:02b},
which circumvents the lemma of Jeulin~\cite{Jeulin:80} mentioned 
above, and an application of the first Ray-Knight theorem 
complete the task.

The rest of the paper is organised as follows. 
Section~\ref{sec:set} describes the setting and states the
main results.  In Section~\ref{sec:pt}
we show how to reduce the main theorems to a problem for Brownian 
motion.  Section~\ref{sec:bc} gives the characterisation for the 
convergence of integral functionals of Brownian motion. 
Sections~\ref{sec:fp} and~\ref{sec:sp} give the two proofs of
Lemma~\ref{lem:bc1}.

%========================================================================================
\section{The Setting and Main Results}
\label{sec:set}
\newpar{}
\label{par:set0}
First we introduce some common notations used in the sequel.
Let us consider an open interval $J=(l,r)\subseteq\bbR$.

\noindent
\begin{itemize}
\item
By $\ol J$ we denote $[l,r]$.

\item
By $L^1_\loc(J)$ we denote the set of Borel functions $J\to[-\infty,\infty]$,
which are locally integrable on $J$, i.e. integrable on compact subsets of~$J$.

\item
For $x\in J$, $L^1_\loc(x)$ denotes the set of Borel functions $f\colon J\to[-\infty,\infty]$
such that $\int_{x-\eps}^{x+\eps}|f(y)|\,\dd y<\infty$ for some $\eps>0$.

\item
Let $\alpha\in[l,r)$, $\beta\in(l,r]$.
By $L^1_\loc(\alpha+)$ we denote the set of Borel functions $f\colon J\to[-\infty,\infty]$
such that $\int_\alpha^z|f(y)|\,\dd y<\infty$ for some $z\in J$, $z>\alpha$.
The notation $L^1_\loc(\beta-)$ is introduced similarly.
\end{itemize}

We will need the following statement. Its proof is straightforward.

\begin{lemma}
\label{lem:set1}
$L^1_\loc(J)=\bigcap_{x\in J}L^1_\loc(x)$.
\end{lemma}

%========================================================================================
\newpar{}
\label{par:set1}
Let the state space be $J=(l,r)$, $-\infty\le l<r\le\infty$, and $Y=(Y_t)_{t\in[0,\infty)}$
be a $J$-valued solution of the one-dimensional SDE
\begin{equation}
\label{eq:set1}
\dd Y_t=\mu(Y_t)\,\dd t+\sigma(Y_t)\,\dd W_t,\quad Y_0=x_0,
\end{equation}
on some filtered probability space $(\Omega,\cF,(\cF_t)_{t\in[0,\infty)},\PP)$,
where $x_0\in J$ and $W$ is an $(\cF_t,\PP)$-Brownian motion.
We allow $Y$ to exit its state space $J$ at a finite time in a continuous way.
The exit time is denoted by $\zeta$.
That is to say, $\PP$-a.s. on $\{\zeta=\infty\}$ the trajectories of $Y$ do not exit~$J$,
while $\PP$-a.s. on $\{\zeta<\infty\}$ we have: either
$\lim_{t\uparrow\zeta}Y_t=r$ or $\lim_{t\uparrow\zeta}Y_t=l$.
Then we need to specify the behaviour of $Y$ after $\zeta$ on $\{\zeta<\infty\}$.
In what follows we assume that on $\{\zeta<\infty\}$ the process $Y$
stays at the endpoint of $J$ where it exits after $\zeta$,
i.e. $l$ and $r$ are by convention absorbing boundaries.

Throughout the paper it is assumed that the coefficients 
$\mu$ and $\sigma$ in~\eqref{eq:set1} satisfy
the Engelbert--Schmidt conditions
\begin{gather}
\label{eq:set2}
\sigma(x)\ne0\;\;\forall x\in J,\\
\label{eq:set3}
\frac1{\sigma^2},\frac\mu{\sigma^2}\in L^1_\loc(J).
\end{gather}
Under \eqref{eq:set2} and~\eqref{eq:set3} SDE~\eqref{eq:set1}
has a weak solution, unique in law, which possibly exits
$J$
(see~\cite{EngelbertSchmidt:85}, \cite{EngelbertSchmidt:91},
or~\cite[Ch.~5, Th.~5.15]{KaratzasShreve:91}).
The Engelbert--Schmidt conditions are reasonable weak assumptions:
any locally bounded Borel function $\mu$
and locally bounded away from~$0$ Borel function $\sigma$ on $J$ 
satisfy \eqref{eq:set2} and~\eqref{eq:set3}.

Finally, the reason for considering an arbitrary interval $J\subseteq\bbR$
as a state space, and not just $\bbR$ itself, is that there are natural examples,
where the Engelbert--Schmidt conditions hold only on a subset of~$\bbR$.
Consider for example geometric Brownian motion
\begin{equation}
\label{eq:set3a}
\dd Y_t=aY_t\,\dd t+bY_t\,\dd W_t,\quad Y_0=x_0>0
\end{equation}
($a,b\in\bbR$, $b\ne0$),
with its natural state space $J=(0,\infty)$.
If we were to take  $J=\bbR$ here, both \eqref{eq:set2} and~\eqref{eq:set3} would
be violated.
Even though we can replace the diffusion coefficient $\sigma(y)=by$ in~\eqref{eq:set3a}
by
$\sigma(y)=by+I_{\{0\}}(y)$, which does not affect solutions of~\eqref{eq:set3a}
but fixes the problem with~\eqref{eq:set2}, the issue with~\eqref{eq:set3}
cannot be resolved in this way. On the other hand, any state space $J=(l,r)$
with $0\le l<x_0<r\le\infty$ is a possible choice when working with SDE~\eqref{eq:set3a};
however, if $l>0$ or $r<\infty$, the convention above
implies we have a stopped geometric Brownian motion.

%========================================================================================
\newpar{}
\label{par:set2}
Now we state some well-known results about the behaviour of one-dimensional diffusions
with the coefficients satisfying the Engelbert--Schmidt conditions
%(see e.g. \cite[Sec.~4.1]{ChernyEngelbert:05}) 
that will be extensively used in the sequel.
Let us also note that these results do not hold beyond the Engelbert--Schmidt conditions.

Let $s$ denote the scale function of $Y$ and $\rho$ the derivative of~$s$, i.e.
\begin{align}
\label{eq:set4}
\rho(x)&=\exp\left\{-\int_c^x \frac{2\mu}{\sigma^2}(y)\,\dd y\right\},\quad x\in J,\\
\label{eq:set5}
s(x)&=\int_c^x \rho(y)\,\dd y,\quad x\in\ol J,
\end{align}
for some $c\in J$. In particular, $s$ is an increasing $C^1$-function $J\to\bbR$
with a strictly positive absolutely continuous derivative,
while $s(r)$ (resp.~$s(l)$) may take value $\infty$ (resp.~$-\infty$).

For $a\in\ol J$ let us define the stopping time
\begin{equation}
\label{eq:set6}
\tau^Y_a=\inf\{t\in[0,\infty)\colon Y_t=a\}\qquad(\inf\emptyset:=\infty).
\end{equation}

\begin{proposition}
\label{prop:set1}
For any $a\in J$ we have $\PP(\tau^Y_a<\infty)>0$.
\end{proposition}

Even though it is assumed 
in Proposition~\ref{prop:set1}
that 
$a\in J$,
we stress that $\tau^Y_a$ is defined for any $a\in\ol J$,
which will be needed in Remark~\ref{rem:set1} below.

Further let us consider the sets
\begin{align*}
A&=\left\{\zeta=\infty,\;\limsup_{t\to\infty}Y_t=r,\;\liminf_{t\to\infty}Y_t=l\right\},\\
B_r&=\left\{\zeta=\infty,\;\lim_{t\to\infty}Y_t=r\right\},\\
C_r&=\left\{\zeta<\infty,\;\lim_{t\uparrow\zeta}Y_t=r\right\},\\
B_l&=\left\{\zeta=\infty,\;\lim_{t\to\infty}Y_t=l\right\},\\
C_l&=\left\{\zeta<\infty,\;\lim_{t\uparrow\zeta}Y_t=l\right\}.
\end{align*}
%\begin{align*}
%A&=\left\{\zeta=\infty,\;\limsup_{t\to\infty}Y_t=r,\;\liminf_{t\to\infty}Y_t=l\right\},\\
%B_r&=\left\{\zeta=\infty,\;\lim_{t\to\infty}Y_t=r\right\},\\
%%C_r&=\left\{\zeta<\infty,\;\lim_{t\uparrow\zeta}Y_t=r\right\},\\
%B_l&=\left\{\zeta=\infty,\;\lim_{t\to\infty}Y_t=l\right\},\\
%C_l&=\left\{\zeta<\infty,\;\lim_{t\uparrow\zeta}Y_t=l\right\}.
%\end{align*}

\begin{proposition}
\label{prop:set2}
Either $\PP(A)=1$ or $\PP(B_r\cup B_l\cup C_r\cup C_l)=1$.
\end{proposition}

\begin{proposition}
\label{prop:set3}
(i) $\PP(B_r\cup C_r)=0$ holds if and only if $s(r)=\infty$.

(ii) $\PP(B_l\cup C_l)=0$ holds if and only if $s(l)=-\infty$.
\end{proposition}

In particular, we get that $\PP(A)=1$ holds if and only if $s(r)=\infty$, $s(l)=-\infty$.

\begin{proposition}
\label{prop:set4}
Assume that $s(r)<\infty$.
Then either $\PP(B_r)>0$, $\PP(C_r)=0$ or $\PP(B_r)=0$, $\PP(C_r)>0$.
Furthermore, we have
$$
\PP\left(\lim_{t\uparrow\zeta}Y_t=r,\;\;Y_t>a\;\forall t\in[0,\zeta)\right)>0
$$
for any $a<x_0$.
\end{proposition}

Propositions~\ref{prop:set1}--\ref{prop:set4} are well-known and 
follow from the Engelbert--Schmidt construction of solutions
(see~e.g.~\cite{EngelbertSchmidt:91} or~\cite[Ch.~5.5]{KaratzasShreve:91})
or can be deduced from the results in~\cite[Sec.~1.5]{EngelbertSchmidt:89}.

\begin{proposition}[Feller's test for explosions]
\label{prop:set5}
We have $\PP(B_r)=0$, $\PP(C_r)>0$ if and only if
$$
s(r)<\infty\quad\text{and}\quad\frac{s(r)-s}{\rho\sigma^2}\in L^1_\loc(r-).
$$
\end{proposition}

Clearly, Propositions \ref{prop:set4} and~\ref{prop:set5},
which contain statements about the behaviour of one-dimensional diffusions
at the endpoint~$r$, have their analogues for the behaviour at~$l$.
Feller's test for explosions in this form is taken from \cite[Sec.~4.1]{ChernyEngelbert:05}.
For a different (but equivalent) form see e.g. \cite[Ch.~5, Th.~5.29]{KaratzasShreve:91}.

%========================================================================================
\newpar{}
\label{par:set3}
In this paper we study convergence of the integral functional
\begin{equation}
\label{eq:set7}
\int_0^t f(Y_u)\,\dd u,\quad t\in[0,\zeta],
\end{equation}
where $f\colon J\to[0,\infty]$ is a nonnegative Borel function.
In this subsection we reduce the study of convergence of~\eqref{eq:set7}
in general to that of convergence of the integral
\begin{equation}
\label{eq:set8}
\int_0^\zeta f(Y_u)\,\dd u
\end{equation}
for a nonnegative Borel function $f\colon J\to[0,\infty]$
such that $\frac f{\sigma^2}\in L^1_\loc(J)$.
In the next subsection we formulate the answer to the latter problem.

Let us consider the set
$$
D=\left\{x\in J\colon \frac f{\sigma^2}\notin L^1_\loc(x)\right\}
$$
and note that $D$ is a closed subset in~$J$.
Let us further define the stopping time
$$
\eta_D=\zeta\wedge\inf\{t\in[0,\infty)\colon Y_t\in D\}\qquad(\inf\emptyset:=\infty).
$$

\begin{theorem}
\label{th:set1}
$\PP$-a.s. we have:
\begin{align}
\label{eq:set9}
\int_0^t f(Y_u)\,\dd u&<\infty,\quad t\in[0,\eta_D),\\
\label{eq:set10}
\int_0^t f(Y_u)\,\dd u&=\infty,\quad t\in(\eta_D,\zeta].
\end{align}
\end{theorem}

\begin{remark}
\label{rem:set1}
After Theorem~\ref{th:set1} it remains only to study the convergence of the integral
$$
\int_0^{\eta_D} f(Y_u)\,\dd u.
$$
If $x_0\in D$, then $\eta_D\equiv0$, and the integral is clearly zero.
Let us assume that $x_0\notin D$ and set
\begin{align*}
\alpha&=\sup(l,x_0)\cap D\qquad(\sup\emptyset:=l),\\
\beta&=\inf(x_0,r)\cap D\qquad(\inf\emptyset:=r).
\end{align*}
It is easy to see that $\eta_D=\tau^Y_\alpha\wedge\tau^Y_\beta$.
Now if we consider $I:=(\alpha,\beta)$ as a new state space for~$Y$,
then $\tau^Y_\alpha\wedge\tau^Y_\beta$ will be the new exit time,
and we will have $\frac f{\sigma^2}\in L^1_\loc(I)$ by Lemma~\ref{lem:set1}.
This concludes the reduction of the study of the convergence of~\eqref{eq:set7}
to that of the convergence of~\eqref{eq:set8}.
\end{remark}

In order to prove Theorem~\ref{th:set1} we need some additional notation.
Since $Y$ is a continuous semimartingale up to the exit time $\zeta$,
one can define its local time $\{L^y_t(Y);y\in J,t\in[0,\zeta)\}$
on the stochastic interval $[0,\zeta)$ for any $y\in J$ in the usual way
(e.g. via the obvious generalization of~\cite[Ch.~VI, Th.~1.2]{RevuzYor:99}).
It follows from Theorem~VI.1.7 in~\cite{RevuzYor:99} that the random field
$\{L^y_t(Y);y\in J,t\in[0,\zeta)\}$ admits a modification such that the map
$(y,t)\mapsto L^y_t(Y)$ is a.s. continuous in $t$ and cadlag in~$y$.\footnote{Moreover,
it can be proved that for a diffusion $Y$ driven by~\eqref{eq:set1} under conditions
\eqref{eq:set2} and~\eqref{eq:set3}, any such modification is, in fact,
a.s. jointly continuous in $(t,y)$; see~\cite[Proposition~A.1]{MijatovicUrusov:10a}.}
As usual we always work with such a modification.
Let us further recall that a.s. on $\{t<\zeta\}$ the function
$y\mapsto L^y_t(Y)$ has a compact support in $J$ and hence
is bounded as a cadlag function with a compact support.

We will need the following result.

\begin{lemma}[Theorem~2.7 in~\cite{ChernyEngelbert:05}]
\label{lem:set2}
Let $a\in J$. Then
$$
L^a_t(Y)>0\quad\text{and}\quad L^{a-}_t(Y)>0\quad\PP\text{-a.s. on }\{\tau^Y_a<t<\zeta\}.
$$
\end{lemma}

\begin{remarks}
\label{rem:set2}
(i) By Proposition~\ref{prop:set1} we have $\PP(\tau^Y_a<\zeta)>0$.
Hence, there exists $t\in(0,\infty)$ such that $\PP(\tau^Y_a<t<\zeta)>0$.

(ii) Let us note that the result of Lemma~\ref{lem:set2} no longer holds
if the coefficients $\mu$ and $\sigma$ of~\eqref{eq:set1} fail to satisfy
the Engelbert--Schmidt conditions (see Theorem~2.6 in~\cite{ChernyEngelbert:05}).
\end{remarks}

\begin{proof}[Proof of Theorem~\ref{th:set1}]
By the occupation times formula, $\PP$-a.s. we have
\begin{equation}
\label{eq:set11}
\int_0^t f(Y_u)\,\dd u=\int_0^t\frac f{\sigma^2}(Y_u)\,\dd\la Y,Y\ra_u
=\int_J\frac f{\sigma^2}(y)L^y_t(Y)\,\dd y,\quad t\in[0,\zeta).
\end{equation}
Then~\eqref{eq:set9} follows from the fact that $\PP$-a.s. on $\{t<\zeta\}$
the function $y\mapsto L^y_t(Y)$ is a cadlag function with a compact support in~$J$.

As for~\eqref{eq:set10}, it immediately follows from~\eqref{eq:set11} and Lemma~\ref{lem:set2}
in the case $x_0\in D$. If $x_0\notin D$, we first observe that $\eta_D=\tau^Y_\alpha\wedge\tau^Y_\beta$
(see Remark~\ref{rem:set1}), and hence
$\{\eta_D<\zeta\}=\{\tau^Y_\alpha<\zeta\}\cup\{\tau^Y_\beta<\zeta\}$.
If $\PP(\tau^Y_\alpha<\zeta)>0$, then, 
since 
$D$
is closed we have
%by the closedness of~$D$, 
$\alpha\in D$
(note that $(l,x_0)\cap D\ne\emptyset$ in this case because otherwise
$\alpha=l$ and $\PP(\tau^Y_\alpha<\zeta)=0$).
Thus, \eqref{eq:set10} on $\{\tau^Y_\alpha<\zeta\}$
follows from~\eqref{eq:set11} and Lemma~\ref{lem:set2}
applied with $a=\alpha\in D$. Similarly we get \eqref{eq:set10} on $\{\tau^Y_\beta<\zeta\}$.
This concludes the proof.
\end{proof}

%========================================================================================
\newpar{}
\label{par:set4}
As pointed out in the previous subsection,
it remains to study the convergence of the integral
\begin{equation}
\label{eq:set12}
\int_0^\zeta f(Y_u)\,\dd u
\end{equation}
for a nonnegative Borel function $f\colon J\to[0,\infty]$ satisfying
\begin{equation}
\label{eq:set13}
\frac f{\sigma^2}\in L^1_\loc(J).
\end{equation}
This study is performed in the following two theorems,
where we separately treat the cases $\PP(A)=1$ and $\PP(B_r\cup B_l\cup C_r\cup C_l)=1$
(see Propositions \ref{prop:set2} and~\ref{prop:set3}).
Below $\nu_L$ denotes the Lebesgue measure on~$J$.

\begin{theorem}
\label{th:set2}
Assume that the function $f\colon J\to[0,\infty]$ satisfies~\eqref{eq:set13}.
Let $s(r)=\infty$ and $s(l)=-\infty$.

(i) If $\nu_L(f>0)=0$, then
$$
\int_0^\zeta f(Y_u)\,\dd u=0\quad\PP\text{-a.s.}
$$

(ii) If $\nu_L(f>0)>0$, then
$$
\int_0^\zeta f(Y_u)\,\dd u=\infty\quad\PP\text{-a.s.}
$$
\end{theorem}

Let us also note that $\zeta=\infty$ $\PP$-a.s. in the case $s(r)=\infty$, $s(l)=-\infty$.

In the remaining case $s(l)>-\infty$ or $s(r)<\infty$ we have
$$
\Omega=\left\{\lim_{t\uparrow\zeta}Y_t=r\right\}\cup\left\{\lim_{t\uparrow\zeta}Y_t=l\right\}\quad\PP\text{-a.s.}
$$
In the following theorem we investigate the convergence of~\eqref{eq:set12}
on $\{\lim_{t\uparrow\zeta}Y_t=r\}$. To this end we need to assume $s(r)<\infty$
because otherwise $\PP(\lim_{t\uparrow\zeta}Y_t=r)=0$ by Proposition~\ref{prop:set3}.

\begin{theorem}
\label{th:set3}
Assume that the function $f\colon J\to[0,\infty]$ satisfies~\eqref{eq:set13}.
Let $s(r)<\infty$.

(i) If
$$
\frac{(s(r)-s)f}{\rho\sigma^2}\in L^1_\loc(r-),
$$
then
$$
\int_0^\zeta f(Y_u)\,\dd u<\infty\quad\PP\text{-a.s. on }\left\{\lim_{t\uparrow\zeta}Y_t=r\right\}.
$$

(ii) If
$$
\frac{(s(r)-s)f}{\rho\sigma^2}\notin L^1_\loc(r-),
$$
then
$$
\int_0^\zeta f(Y_u)\,\dd u=\infty\quad\PP\text{-a.s. on }\left\{\lim_{t\uparrow\zeta}Y_t=r\right\}.
$$
\end{theorem}

Clearly, Theorem~\ref{th:set3} has its analogue that describes the convergence
of~\eqref{eq:set12} on ${\{\lim_{t\uparrow\zeta}Y_t=l\}}$.

%========================================================================================
\section{Proofs of Theorems \protect\ref{th:set2} and~\protect\ref{th:set3}}
\label{sec:pt}
In this section we prove Theorems \ref{th:set2} and~\ref{th:set3}.
In the latter proof we apply Lemma~\ref{lem:bc1} below,
which will be proved in the next sections.

Let us set
\begin{equation}
\label{eq:pt1}
\wt Y_t=s(Y_t),\quad t\in[0,\zeta).
\end{equation}
Then
\begin{equation}
\label{eq:pt2}
d\wt Y_t=\wt\sigma(\wt Y_t)\,\dd W_t,\quad t\in[0,\zeta),
\end{equation}
where
$$
\wt\sigma(x)=(\rho\sigma)\circ s^{-1}(x),\quad x\in(s(l),s(r)).
$$
In particular, $\wt Y$ is a continuous local martingale
on the stochastic interval $[0,\zeta)$.
By the Dambis--Dubins--Schwarz theorem,
there exists a Brownian motion $B$ starting from $s(x_0)$
(possibly on an enlargement of the initial probability space) such that
\begin{equation}
\label{eq:pt3}
\wt Y_t=B_{\la\wt Y,\wt Y\ra_t}\quad\PP\text{-a.s.},\quad t\in[0,\zeta).
\end{equation}
Let us also introduce the function
$$
\wt f(x)=f\circ s^{-1}(x),\quad x\in(s(l),s(r)).
$$

\begin{proof}[Proof of Theorem~\ref{th:set2}]
Here $s(r)=\infty$ and $s(l)=-\infty$. Hence $\zeta=\infty$ $\PP$-a.s. and, moreover, $\PP(A)=1$
(see Propositions \ref{prop:set2} and~\ref{prop:set3}). Then \eqref{eq:pt1} implies that
$$
\PP\left(\limsup_{t\to\infty}\wt Y_t=\infty,\;\liminf_{t\to\infty}\wt Y_t=-\infty\right)=1.
$$
Now it follows from \eqref{eq:pt3} that $\la\wt Y,\wt Y\ra_\infty=\infty$ $\PP$-a.s.
We have
\begin{align*}
\int_0^\infty f(Y_u)\,\dd u
&=\int_0^\infty \wt f(\wt Y_u)\,\dd u
=\int_0^\infty \frac{\wt f}{\wt\sigma^2}\left(B_{\la\wt Y,\wt Y\ra_u}\right)\,\dd\la\wt Y,\wt Y\ra_u\\
&=\int_0^\infty \frac{\wt f}{\wt\sigma^2}(B_v)\,\dd v
=\int_\bbR \frac{\wt f}{\wt\sigma^2}(x)L^x_\infty(B)\,\dd x\quad\PP\text{-a.s.}
\end{align*}
The first equality above is clear (we used $\zeta=\infty$ $\PP$-a.s.),
the second follows from \eqref{eq:pt2} and~\eqref{eq:pt3},
the third is due to the continuity of $\la\wt Y,\wt Y\ra$ and the fact that
$\la\wt Y,\wt Y\ra_\infty=\infty$ $\PP$-a.s.,
and the last one follows from the occupation times formula
($L^x_t(B)$ denotes the local time of the Brownian motion $B$ at time~$t$
and at level~$x$). It remains to note that $\nu_L(f>0)>0$ is equivalent to $\nu_L(\wt f>0)>0$
and that for a Brownian local time,
$\PP$-a.s. it holds $L^x_\infty(B)\equiv\infty$ $\forall x\in\bbR$
(see e.g. \cite[Ch.~VI, \S~2]{RevuzYor:99}).
The proof is completed.
\end{proof}

\begin{proof}[Proof of Theorem~\ref{th:set3}]
Here $s(r)<\infty$, i.e. $\PP(\lim_{t\uparrow\zeta}Y_t=r)>0$.
let us set $R:=\{\lim_{t\uparrow\zeta}Y_t=r\}$ and observe that~\eqref{eq:pt1}
and~\eqref{eq:pt3}
imply
$$
R\equiv\left\{\lim_{t\uparrow\zeta}Y_t=r\right\}
=\left\{\lim_{t\uparrow\zeta}\wt Y_t=s(r)\right\}
=\left\{\lim_{t\uparrow\zeta}B_{\la\wt Y,\wt Y\ra_t}=s(r)\right\}.
$$
In particular,
\begin{equation}
\label{eq:pt4}
\la\wt Y,\wt Y\ra_\zeta=\tau^B_{s(r)}\quad\PP\text{-a.s. on }R,
\end{equation}
where $\tau^B_{s(r)}$ denotes the hitting time of the level $s(r)$
by the Brownian motion~$B$.
Let us note that $\zeta$ may be finite or infinite on~$R$
(see Propositions \ref{prop:set4} and~\ref{prop:set5} for details),
but it follows from~\eqref{eq:pt4} that $\la\wt Y,\wt Y\ra_\zeta$
is in either case finite on~$R$.
Similarly to the previous proof we get
\begin{equation}
\label{eq:pt5}
\int_0^\zeta f(Y_u)\,\dd u=\int_0^{\tau^B_{s(r)}}
\frac{\wt f}{\wt\sigma^2}(B_v)\,\dd v\quad\PP\text{-a.s. on }R.
\end{equation}
The question of convergence of the integral in the right-hand side of~\eqref{eq:pt5}
is studied in Lemma~\ref{lem:bc1} below.
It is easy to obtain from~\eqref{eq:set13} that $\frac{\wt f}{\wt\sigma^2}\in L^1_\loc(s(J))$,
which means that Lemma~\ref{lem:bc1} can be applied (see~\eqref{eq:bc1}).
Thus, to study the convergence of the integral in the right-hand side of~\eqref{eq:pt5}
we need to check whether
$$
(s(r)-x)\frac{\wt f}{\wt\sigma^2}(x)\in L^1_\loc(s(r)-)
$$
(the notation ``$f(x)\in\fM$'' for a function $f$ and a class of functions $\fM$
is understood to be synonymous to
``$f\in\fM$''). We have
$$
\int_{s(\cdot)}^{s(r)}\frac{(s(r)-x)f(s^{-1}(x))}{\rho^2(s^{-1}(x))\sigma^2(s^{-1}(x))}\,\dd x
=\int_\cdot^r\frac{(s(r)-s(y))f(y)}{\rho(y)\sigma^2(y)}\,\dd y
$$
(recall that $s'=\rho$). Now the statement of Theorem~\ref{th:set3}
follows from~\eqref{eq:pt5} and Lemma~\ref{lem:bc1}.
\end{proof}

%========================================================================================
\section{The Setting and Notation in the Brownian Case}
\label{sec:bc}
It remains to prove Lemma~\ref{lem:bc1} below.
From now on let us consider a Brownian motion $B$ starting from $x_0\in\bbR$.
We will extensively use the notation $\tau^B_a$ ($a\in\bbR$) for the stopping time
defined as in~\eqref{eq:set6}. Below we use the notation ``$f(x)\in\fM$''
for a function $f$ and a class of functions $\fM$ as a synonym for ``$f\in\fM$''.

\begin{lemma}
\label{lem:bc1}
Let $B$ be a Brownian motion starting from $x_0\in\bbR$ and $x_0<r<\infty$.
Assume that the function $f\colon I\to[0,\infty]$ with $I:=(-\infty,r)$ satisfies
\begin{equation}
\label{eq:bc1}
f\in L^1_\loc(I).
\end{equation}

(i) If $(r-x)f(x)\in L^1_\loc(r-)$, then
$$
\int_0^{\tau^B_r}f(B_u)\,\dd u<\infty\quad\PP\text{-a.s.}
$$

(ii) If $(r-x)f(x)\notin L^1_\loc(r-)$, then
$$
\int_0^{\tau^B_r}f(B_u)\,\dd u=\infty\quad\PP\text{-a.s.}
$$
\end{lemma}

In Sections~\ref{sec:fp} and~\ref{sec:sp} we give two different proofs of Lemma~\ref{lem:bc1}.

%========================================================================================
\section{First Proof of Lemma~\protect\ref{lem:bc1}}
\label{sec:fp}
This method is based on Williams' theorem (see~\cite[Ch.~VII, Cor.~4.6]{RevuzYor:99})
and Cherny's investigation of convergence of integral functionals of Bessel processes
(see~\cite{Cherny:01}).

By the occupation times formula and~\eqref{eq:bc1}, $\PP$-a.s. we get
\begin{equation}
\label{eq:fp1}
\int_0^t f(B_u)\,\dd u<\infty,\quad t\in[0,\tau^B_r).
\end{equation}
By $\rho=(\rho_t)_{t\in[0,\infty)}$ we denote a three-dimensional Bessel process starting from~0.
Let us set
$$
\xi=\sup\{t\in[0,\infty)\colon \rho_t=r-x_0\}
$$
(note that $\xi$ is a finite random variable because $\rho_t\to\infty$~a.s.).
By Williams' theorem,
\begin{equation}
\label{eq:fp2}
\Law\left(r-B_{\tau^B_r-t};t\in[0,\tau^B_r)\right)
=\Law\left(\rho_t;t\in[0,\xi)\right),
\end{equation}
where ``$\Law$'' means distribution.
It follows from Theorem~2.2 in~\cite{Cherny:01} that,
for a nonnegative function~$g$,

\smallskip\noindent
(A) $xg(x)\in L^1_\loc(0+)$ implies that a.s. it holds
$$
\exists\eps>0\;\;\int_0^\eps g(\rho_u)\,\dd u<\infty;
$$

\smallskip\noindent
(B) $xg(x)\notin L^1_\loc(0+)$ implies that a.s. it holds
$$
\forall\eps>0\;\;\int_0^\eps g(\rho_u)\,\dd u=\infty.
$$

\noindent
By \eqref{eq:fp1}, \eqref{eq:fp2} and (A),~(B),
the question reduces to whether $xf(r-x)\in L^1_\loc(0+)$,
or, equivalently, to whether $(r-x)f(x)\in L^1_\loc(r-)$.
This concludes the proof.

%========================================================================================
\section{Second Proof of Lemma~\protect\ref{lem:bc1}}
\label{sec:sp}
We take the idea for this proof from Theorem~1.4 in Delbaen and Shirakawa~\cite{DelbaenShirakawa:02b}.
The method is based on the first Ray-Knight theorem (see \cite[Ch.~XI, Th.~2.2]{RevuzYor:99}).

By the occupation times formula,
\begin{equation}
\label{eq:sp1}
\int_0^{\tau^B_r} f(B_u)\,\dd u %=\int_{-\infty}^\infty f(x)L^x_{\tau^B_r}(B)\,\dd x
=\int_{-\infty}^r f(x)L^x_{\tau^B_r}(B)\,\dd x\quad\PP\text{-a.s.}
\end{equation}
Since $\PP$-a.s. the mapping $x\mapsto L^x_{\tau^B_r}(B)$ is a continuous function with
a compact support in~$\bbR$, we get from~\eqref{eq:bc1} that
$$
\int_{-\infty}^{x_0} f(x)L^x_{\tau^B_r}(B)\,\dd x<\infty\quad\PP\text{-a.s.}
$$
By~\eqref{eq:sp1}, the question of whether $\int_0^{\tau^B_r} f(B_u)\,\dd u$ is finite
reduces to the question of whether $\int_{x_0}^r f(x)L^x_{\tau^B_r}(B)\,\dd x$ is finite,
or, equivalently, to 
\begin{equation}
\label{eq:sp2}
\text{whether}\quad
\int_0^{r-x_0}f(r-u)L^{r-u}_{\tau^B_r}(B)\,\dd u
\quad\text{is finite.}
\end{equation}

Let $\ol W$ and $\wt W$ be independent Brownian motions starting from~0.
Let us set
\begin{equation}
\label{eq:sp3}
\eta_t=\ol W_t^2+\wt W_t^2,
\end{equation}
i.e. $\eta=(\eta_t)_{t\in[0,\infty)}$ is a squared two-dimensional Bessel process starting from~0.
It follows from the first Ray-Knight theorem that
\begin{equation}
\label{eq:sp4}
\Law\left(L^{r-u}_{\tau^B_r};u\in[0,r-x_0]\right)
=\Law\left(\eta_u;u\in[0,r-x_0]\right).
\end{equation}
In what follows we prove that, for a Brownian motion $W$ starting from~0,

\smallskip\noindent
(A) $xf(r-x)\in L^1_\loc(0+)$ implies that
$$
\int_0^{r-x_0}f(r-u)W_u^2\,\dd u<\infty\quad\text{a.s.};
$$

\smallskip\noindent
(B) $xf(r-x)\notin L^1_\loc(0+)$ implies that
$$
\int_0^{r-x_0}f(r-u)W_u^2\,\dd u=\infty\quad\text{a.s.}
$$

\noindent
Together with \eqref{eq:sp2}--\eqref{eq:sp4} this will complete the proof of Lemma~\ref{lem:bc1}.

By Fubini's theorem,
$$
\EE\int_0^{r-x_0}f(r-u)W_u^2\,\dd u=\int_0^{r-x_0}f(r-u)u\,\dd u,
$$
so (A) is immediate (recall that \eqref{eq:bc1} holds).

In order to prove (B) we assume that
\begin{equation}
\label{eq:sp5}
\PP\left(\int_0^{r-x_0}f(r-u)W_u^2\,\dd u<\infty\right)>0.
\end{equation}
Then there exists a sufficiently large $M<\infty$ such that
$$
\gamma:=\PP(R)>0,\quad\text{where}\quad
R:=\left\{\int_0^{r-x_0}f(r-u)W_u^2\,\dd u\le M\right\}.
$$
%where
%$$
%R:=\left\{\int_0^{r-x_0}f(r-u)W_u^2\,\dd u\le M\right\}.
%$$
Let us note that, for any positive $\delta$ and $u$, the probability
$$
\PP(W_u^2\ge\delta^2u)=\PP(|W_u/\sqrt{u}|\ge\delta)=\PP(|N(0,1)|\ge\delta)
$$
does not depend on~$u$. We take a sufficiently small $\delta>0$ such that
$$
\PP(|N(0,1)|\ge\delta)\ge1-\frac\gamma2.
$$
Then, for any~$u$,
$$
\EE(W_u^2I_R)\ge\delta^2u\PP(R\cap\{W_u^2\ge\delta^2u\})\ge\frac\gamma2\delta^2u.
$$
By Fubini's theorem,
$$
\EE\left[I_R\int_0^{r-x_0}f(r-u)W_u^2\,\dd u\right]
=\int_0^{r-x_0}f(r-u)\EE(W_u^2I_R)\,\dd u
\ge\frac\gamma2\delta^2\int_0^{r-x_0}f(r-u)u\,\dd u.
$$
The left-hand side is finite as on the event $R$ the integral is not greater than~$M$.
Thus, \eqref{eq:sp5} implies $uf(r-u)\in L^1_\loc(0+)$, which proves~(B)
and completes the proof of Lemma~\ref{lem:bc1}.

%========================================================================================
%\nocite{*}
%\bibliographystyle{chicago}
%\bibliographystyle{amsplain}
%\bibliographystyle{plain}
\bibliographystyle{abbrv}
\bibliography{refs}
\end{document}